\documentclass[11pt,reqno]{amsart}

\usepackage{amssymb,latexsym,amsrefs,enumitem}
\usepackage{graphicx}
\usepackage{url}		%does nice formatting of URLs

\calclayout

\theoremstyle{plain}
\numberwithin{equation}{section}
\newtheorem{thm}{Theorem}[section]
\newtheorem{theorem}[thm]{Theorem}
\newtheorem{lemma}[thm]{Lemma}

\newtheorem{proposition}[thm]{Proposition}

\begin{document}
%% replace the values in the next three lines by the correct information
%\monthyear{Month Year}
%\volnumber{Volume, Number}
%\setcounter{page}{1}

\title{Extension of the Equation $\sum\limits_{j=1}^{k}jF_{j}^{p}=F_{n}^{q}$ to a Family of Lucas Sequences}
\author{Benjamin Earp-Lynch}
\author{Simon Earp-Lynch}
\address{Department of Mathematics\\
		Carleton University\\
		K1S 5B6, Canada}
\email{benjaminearplynch@cmail.carleton.ca}
\email{simonearplynch@cmail.carleton.ca}
\author{Omar Kihel}
\address{Department of Mathematics and Statistics\\
		Brock University\\
		L2S 3A1, Canada}
\email{okihel@brocku.ca}
\thanks{The third and fourth authors were supported by King Mongkut's Institute of Technology Ladkrabang Research Fund (KREF206502). The third author was supported in part by NSERC}
\author{Puntani Pongsumpun}
\address{Department of Mathematics, School of Science\\
		King Mongkut's Institute of Technology Ladkrabang\\
		Bangkok 10520, Thailand}
\email{puntani.po@kmitl.ac.th}

\begin{abstract}
We solve the equation $\sum\limits_{j=1}^{k}jU_{j}(x,y)^{p}=U_{n}(x,y)^{q}$ positive integers $x,p,q,k,n$, with $y=\pm1$ and $\max\{p,q\}\leq11$, where $U_{m}(x,y)=\frac{\alpha^{m}-\beta^{m}}{\alpha-\beta}$ for $\alpha$ and $\beta$ roots of the polynomial $t^2-xt+y$. This generalizes existing results on similar equations, wherein the sequence was fixed as either the Fibonacci or Pell numbers. In addition, we find all solutions with $k=2$ and $y=\pm1$.
\end{abstract}

\maketitle

\section{Introduction}

Let $F_{n}$ denote the $n$th term in the sequence of Fibonacci numbers, which may be defined using Binet's formula $$F_{n}=\frac{\alpha^{n}-\beta^{n}}{\alpha-\beta},$$ where $\alpha$ and $\beta$ are the roots of $t^{2}-t-1$. 
In \cite{N-S-S}, Soydan, N\'{e}meth and Szalay examine the equation

\begin{equation}\label{FibonacciEq}
\sum_{j=1}^{k}jF_{j}^{p}=F_{n}^{q}.
\end{equation}

They found  $$(n,k,p,q)\in\{(1,1,p,q),(2,1,p,q),(4,2,p,1),(4,3,1,2),(8,4,1,1)\}$$ to be the only solutions with $n,k$ positive integers and $p,q$ in the set $\{1,2\}$. Defining the solutions $F_{1}^{p}=F_{1}^{q}=F_{2}^{q}$ and $F_{1}^{p}+2F_{2}^{p}=F_{4}$ to be trivial, they then conjecture that the only nontrivial solutions in positive integers to \eqref{FibonacciEq} are $(n,k,p,q)=(4,3,1,2),(8,4,1,1),$ and $(4,3,3,3)$.

Further progress on equation \eqref{FibonacciEq} has been made in \cites{Luca1,Luca2}. In \cite{Luca1}, the authors explicitly find all solutions with $\max\{p,q\}\leq10$, thereby verifying the conjecture made in \cite{N-S-S} for these exponents. Their method would yield a bound for $k$ given any fixed exponents $p$ and $q$. In \cite{Luca2}, the authors use Baker's method of linear forms in logarithms to find a bound on the size of the parameters. Specifically, they find that $\max{\{(n,k,p,q)\}}<10^{2500}$ in any potential solution of equation \eqref{FibonacciEq}, which shows that the number of solutions is certainly finite. The bound is, however, incredibly large, and the authors note that a full solution would be far beyond the reach of present-day computing power with this method alone.

Let $L_{n}$ denote the $n$th Lucas number, which may be defined by $$L_{n}=\alpha^{n}+\beta^{n}.$$ Solutions to the equation $$\sum\limits_{j=1}^{k}jF_{j}^{p}=L_{n}^{q}$$ were examined in \cite{N-T-T}, where the maximum of the parameters $(n,k,p,q)$ was bounded using methods similar to those in \cite{Luca2}. Again, the bound is enormous.

The Pell numbers, the $n$th of which we shall denote $P_{n}$, may be defined using a similar Binet-type formula with $\alpha$ and $\beta$ roots of the polynomial $t^{2}-2t-1$. 
In \cite{T-T}, the authors extend the results of \cite{N-S-S} to Pell numbers,  finding that the only solutions to the equation
\begin{equation}\label{PellNumEq}
\sum_{j=1}^{k}jP_{j}^{p}=P_{n}^{q}
\end{equation}
with exponents in the set $\{1,2\}$ are $(n,k,p,q)=(1,1,p,q)$, which they call trivial, and $(3,2,1,1)$.

\vspace*{2mm}

Let $x$ and $y$ be integers and $\alpha$ and $\beta$ be the roots of the polynomial $f(t)=t^{2}-xt+y$. We may define the Lucas sequence of the first kind, $\{U_{n}(x,y)\}_{n\in\mathbb{Z}}$, using the Binet formula
\begin{equation}\label{FirstKindBinet}
U_{n}(x,y)=\frac{\alpha^{n}-\beta^{n}}{\alpha-\beta},
\end{equation}
or using the recurrence $$U_{0}(x,y)=0,\quad U_{1}(x,y)=1,\text{ and}\quad U_{n+1}(x,y)=xU_{n}-yU_{n-1}.$$ Both of these definitions may be extended so as to define $U_{n}$ for $n$ a negative integer, and for $n>0$, we find that $U_{-n}=\frac{-U_{n}}{y^{n}}$. Observe that when $x=1$ and $y=-1$, this is the Fibonacci sequence, that is $U_{n}(1,-1)=F_{n}$ for all $n$. 
Interesting sequences also arise when fixing $y=1$, for instance $U_{n}(3,1)=F_{2n}$ is the sequence of even-indexed Fibonacci numbers. The values $(v,u)=\left(\pm\left(\frac{U_{n+1}(4,1)-U_{n-1}(4,1)}{2}\right),\pm U_{n}(4,1)\right)$ are the solutions to the Pell equation $v^{2}-3u^{2}=1$, in other words the units in the ring of integers of the number field $\mathbb{Q}(\sqrt{3})$. 
Defined recursively, the sequence $U_{n}(2,1)=n$ is simply the sequence of integers. 
\vspace*{2mm}

The aim of this paper shall be to extend the results in \cites{N-S-S,Luca1,T-T} to all sequences $\{U_{n}(x,y)\}$ for $x$ any positive integer and $y=\pm1$. Specifically, we prove the following theorem.
\begin{theorem}\label{MainThm}
The only solutions to the equation $1+2x^p=U_{n}(x,y)^{q}$ in positive integers $n,p,q,x$ and with $y=\pm1$ are $$(n,p,q,x,y)=(4,p,1,1,-1), (1+2^{p+1},p,1,2,1) ,(3,2,2,2,1) ,(3,1,1,2,-1) \text{ and } (5,3,1,3,1).$$ Moreover, suppose that $n,k,p,q,x$ are positive integers with $\max\{p,q\}\leq11$, $k\geq3$, and that $y=\pm1$, with $x\geq2$ when $y=-1$ and $x\geq3$ when $y=1$. Then the equation
\begin{equation}\label{MainEq}
\sum_{j=1}^{k}jU_{j}(x,y)^{p}=U_{n}(x,y)^{q}
\end{equation}
has no solutions.
\end{theorem}
Section \ref{Prelims} will establish useful facts and results needed in order to prove this Theorem \ref{MainThm}. The first part of the theorem, concerning the solutions in the case $k=2$, without the need for a bound on the exponents $p$ and $q$, will be proven in section \ref{k2Case}. Section \ref{MainSec} will finish the proof, generalizing the results of \cites{N-S-S,Luca1,T-T} to sequences of the form $U_{n}(x,\pm1)$. There will be a brief explanation of our decision to limit the scope of this paper to sequences with $y=\pm1$ at the end of sections \ref{k2Case} and \ref{MainSec}, by which point the reasons for this choice will be more evident.

The solutions in the case $x=1,y=-1$ have been established in \cite{Luca1} for $\max\{p,q\}\leq10$, and to extend their results to $p=11$ or $q=11$ requires only that one change the constants in their final inequality and then use a computer to search for solutions up to the resulting bound on $k$.

In the case $x=2,y=1$, where $U_{n}(2,1)=n$, equation \eqref{MainEq} becomes $$1^{p+1}+2^{p+1}+\dots+k^{p+1}=n^{q},$$ which is the so-called general cannonball equation, first proposed in the case corresponding to $p=1,q=2$ by, coincidentally, \'Edouard Lucas \cite{Lucas3}, and then investigated in the case of general exponents $p$ and $q$ in \cite{Schaff}, where the solutions with $p\leq10$ were established for many different values of $q$. A few more recent results on the equation include \cites{MR3349441,MR2008104,B-G-P}.

Finally, since the sequence $U_{n}(1,1)$ is periodic with only the terms $1,0$ and $-1$, this case may be solved by inspection. The left side of equation \eqref{MainEq} equals $1$ when $k=1$ and $p$ is any integer, and when $k$ is $1$ or $2$ modulo $6$ and $p$ is odd, and equals $-1$ when $k$ is $4$ or $5$ modulo $6$ and $p$ is odd. The right side of \eqref{MainEq} equals $1$ when $n$ is $1$ or $2$ modulo $6$ and $q$ is any integer, and when $n$ is $4$ or $5$ modulo $6$ and $q$ is even, and equals $-1$ when $n$ is $4$ or $5$ modulo $6$ and $q$ is odd.

The methods that follow will be similar to those used in \cite{Luca1}, but the more general nature of our result will require that our methods diverge at times. 

\section{Preliminaries}\label{Prelims}

In this section, we shall state some facts upon which the rest of our work will rely. Proofs will be necessary in some cases, so that we may apply certain established results concerning Fibonacci numbers to a more general class of Lucas sequences.

Throughout the rest of this paper, we will use $\alpha$ and $\beta$ to denote the roots of $f(t)=t^{2}-xt+y$. In short order we will restrict the values of $x$ and $y$, but we will establish a few of our initial identities more generally first since we may do so with little extra complication. Let $$\alpha=\frac{x+\sqrt{D}}{2}\quad\text{ and }\quad\beta=\frac{x-\sqrt{D}}{2},\quad D=(\alpha-\beta)^2=x^{2}-4y.$$ We will let $U_{n}(x,y)$ denote the Lucas sequence of the first kind as in \eqref{FirstKindBinet}, and $V_{n}(x,y)$ denote the Lucas sequence of the second kind, defined by
\begin{equation}\label{SecondKindBinet}
V_{n}(x,y)=\alpha^{n}+\beta^{n}.
\end{equation} 
When the context permits it, we will use $U_{n}$ (and likewise $V_{n}$) instead of $U_{n}(x,y)$ (resp. $V_{n}(x,y)$) to refer to an arbitrary sequence under the most recently given restrictions on the values of $x$ and $y$. 

\vspace*{2mm}

We state some well-known identities in the following lemma, that we may refer to them later. We remark that these identities hold for all Lucas sequences of the first kind, regardless of the choice of $x$ and $y$. Some of them, including identities \ref{id3} and \ref{id4}, may be found in Chapter 5 of \cite{B-R}.
\begin{lemma}\label{Identities}
The following hold for all Lucas sequences $U_{n}(x,y)$, all integers $m,n$ and positive integers $r$:
\begin{enumerate}[label=\textnormal{(\roman*)}]
\item{} $\alpha^{n}=\alpha U_{n}-yU_{n-1}$\label{id1};
\item{} $\beta^{n}=\beta U_{n}-yU_{n-1}$\label{id2};
\item{} $V_{n}=U_{n+1}-yU_{n-1}$\label{id1a};
\item{} $U_{m+n}=U_{m}V_{n}-y^{n}U_{m-n}$\label{id3};
\item{} $U_{n}^{2}-U_{n-1}U_{n+1}=y^{n-1}$\label{id3a};
\item{} $V_{n}^{2}=DU_{n}^{2}+4y^{n}$\label{id4};
\item{} $U_{r}=\sum\limits_{i=0}^{\lfloor\frac{r-1}{2}\rfloor}(-y)^{i}\binom{r-i-1}{i}x^{r-2i-1}$\label{id5};
\end{enumerate}
\end{lemma}
\begin{proof}
The first six identities may be shown by straightforward application of the Binet formula $U_{n}=\frac{\alpha^{n}-\beta^{n}}{\alpha-\beta}$. We prove \ref{id5} by induction. Observe that
\begin{align*}
U_{1}=&\sum\limits_{i=0}^{0}(-y)^{0}\binom{1-1}{0}x^{1-1}=x^{0}=1,\text{ and}\\
U_{2}=&\sum\limits_{i=0}^{0}(-y)^{0}\binom{2-1}{0}x^{2-1}=x^1=x.
\end{align*}
Now suppose that identity \ref{id5} holds for all positive integers $m\leq r$. Then
\begin{align*}
U_{r+1}=xU_{r}-yU_{r-1}&=x\sum\limits_{i=0}^{\lfloor{\frac{r-1}{2}}\rfloor}(-y)^{i}\binom{r-i-1}{i}x^{r-2i-1}-y\sum\limits_{i=0}^{\lfloor{\frac{r-2}{2}}\rfloor}(-y)^{i}\binom{r-i-2}{i}x^{r-2i-2}\\
&=x^{r}+\sum\limits_{i=1}^{\lfloor{\frac{r-1}{2}}\rfloor}(-y)^{i}\binom{r-i-1}{i}x^{r-2i}+\sum\limits_{i=1}^{\lfloor{\frac{r-2}{2}}\rfloor+1}(-y)^{i}\binom{r-i-1}{i-1}x^{r-2i}\\
&=\begin{cases}x^{r}+\sum\limits_{i=1}^{\lfloor{\frac{r}{2}}\rfloor}(-y)^{i}\binom{r-i}{i}x^{r-2i}\text{ if $r$ is odd,}\\
x^{r}+\sum\limits_{i=1}^{\lfloor{\frac{r}{2}}\rfloor-1}(-y)^{i}\binom{r-i}{i}x^{r-2i}+(-y)^{\lfloor \frac{r}{2}\rfloor}\binom{r-\lfloor\frac{r}{2}\rfloor-1}{\lfloor\frac{r}{2}\rfloor-1}x^{r-2\lfloor\frac{r}{2}\rfloor}\text{ $r$ even,}
\end{cases}\\
&=\sum\limits_{i=0}^{\lfloor{\frac{r}{2}}\rfloor}(-y)^{i}\binom{r-i}{i}x^{r-2i},
\end{align*} 
which proves the identity for all positive integer values of $r$.
\end{proof}
Now assume that $x\geq1$ and $|y+1|\leq x$. Some of the consequences of these assumptions are that:
\begin{itemize}
\item{} $D=x^2-4y\geq x-2\geq0$ when $x\geq2$, and since the condition $|y+1|\leq x$ means that $y=0,-1$ or $-2$ when $x=1$, it follows that $D\geq0$ in these cases as well and equality holds only when $x=2,y=1$.
\item{} $\alpha=\frac{x+\sqrt{x^{2}-4y}}{2}\geq \begin{cases}x-1&\text{if $y>0$}\\ x&\text{if $y\leq0$}\end{cases},$ and since $y=0,-1$ or $-2$ when $x=1$, in either case we have $\alpha\geq1$.
\item{} $-1=\frac{x-\sqrt{x^{2}-4(-x-1)}}{2}\leq\beta\leq\frac{x-\sqrt{x^{2}-4(x-1)}}{2}=1$, and so $|\beta|<1$. 
\item{} The sequences $U_{n}(x,y)$ and $V_{n}(x,y)$ are nonnegative and increasing (in the case of $V_{n}$, increasing for $n>0$). When $y<0$ this is clear from the initial values and the recurrence relation. When $y\geq0$, this can be observed from the fact that $U_{n+1}(x,y)\geq (y+1)U_{n}(x,y)-yU_{n-1}(x,y)\geq y(U_{n}-U_{n-1})+U_{n}$, and so the result follows by induction. (Likewise for $V_{n}(x,y)$.)
\end{itemize}
We briefly consider the case when $y=0$. In this case we have $U_{n}(x,y)=x^{n-1}$ for $n\geq1$. This means equation \eqref{MainEq} has the solution $k=1,n=1$ for all $p$ and $q$. If $x=1$ then \eqref{MainEq} becomes $\frac{k(k+1)}{2}=1$, which has no solutions for $k>1$. If $x>1$ then equation \eqref{MainEq} becomes $\sum\limits_{j=1}^{k}jx^{p(j-1)}=x^{q(n-1)}$. The left hand side of this equation is $1\pmod{x}$ and, when $n>1$, the right hand side is $0\pmod{x}$ so there are no solutions with $x>1$ and $n>1$.
\\
Using the established notation, and the preceding lemma, we formulate bounds on the terms in these sequences. Our restrictions on $x$ and $y$ allow us to approximate the rate of growth of the terms in these sequences by powers of $\alpha$. This is an extension of Lemma 3 in \cite{Luca1}.

\begin{lemma}\label{UVBoundsNew}
For any positive integer $x$ and integer $y$ with $|y+1|\leq x$ such that $D=x^2-4y>0$, the following bounds apply for all integers $n\geq2$.
\begin{align}
x\alpha^{n-2}\leq& U_{n}(x,y)\leq\begin{cases} \alpha^{n-1}\text{ if $y\leq0$,}\\\frac{\alpha^{n}}{\sqrt{D}}\text{ if $y\geq0$,}\end{cases}\label{UBoundsNew}\\
x\alpha^{n-1}\leq& V_{n}(x,y)\leq\frac{\alpha^{n+1}}{x}\text{ when $y\leq0$,}\label{VBoundsNew-}\\
\alpha^{n}\leq& V_{n}(x,y)\leq\frac{\alpha^{n+1}}{\sqrt{D}}\text{ when $y\geq0$.}\label{VBoundsNew+}
\end{align}
\end{lemma}

\begin{proof}
When $n\geq2$, we have $y^2\alpha^{n-4}=\beta^2\alpha^{n-2}\geq\beta^{2}\geq\beta^{n}$. So using the facts that $\alpha\beta=y$ and $\alpha+\beta=x$, we obtain $$U_{n}=\frac{\alpha^{n}-\beta^{n}}{\alpha-\beta}\geq\frac{\alpha^{n}-y^2\alpha^{n-4}}{\alpha-\beta}=\left(\alpha^{n-1}+y\alpha^{n-3}\right)\left(\frac{\alpha-\beta}{\alpha-\beta}\right)=(\alpha+\beta)\alpha^{n-2}=x\alpha^{n-2},$$ which proves the left inequality in \eqref{UBoundsNew}.

From Lemma \ref{Identities}\ref{id1}, we have $\alpha^{n-1}=U_{n}-y\frac{U_{n-1}}{\alpha}$, so $U_{n}\leq\alpha^{n-1}$ when $y\leq0$. When $y\geq0$, we have $\beta\geq0$, and so $U_{n}=\frac{\alpha^{n}-\beta^{n}}{\sqrt{D}}\leq\frac{\alpha^{n}}{\sqrt{D}}$, which proves the right inequality in \eqref{UBoundsNew}.

To show the inequalities in \eqref{VBoundsNew-} and \eqref{VBoundsNew+}, using the left inequality in \eqref{UBoundsNew} and Lemma \ref{Identities}\ref{id1a}, we see that when $y\leq0$, we have $x\alpha^{n-1}\leq U_{n+1}\leq U_{n+1}-yU_{n-1}=V_{n}$, which proves the left inequality in \eqref{VBoundsNew-}, and since $-\beta\leq-\frac{\alpha\beta}{x}=\frac{-y}{x}$, we have $V_{n}=\alpha^{n}+\beta^{n}\leq\alpha^{n}-\beta\leq\alpha^{n}-\frac{y}{x}\leq\alpha^{n}-y\frac{\alpha^{n-1}}{x}=\frac{\alpha^{n+1}}{x}$, giving the right inequality in \eqref{VBoundsNew-}. Similarly, when $y\geq0$, we have $\beta\geq0$ and so $V_{n}=\alpha^{n}+\beta^{n}\geq\alpha^{n}$, giving the left inequality in \eqref{VBoundsNew+}, and by Lemma \ref{Identities}\ref{id1a} and \eqref{UBoundsNew} $V_{n}=U_{n+1}-yU_{n-1}\leq\frac{\alpha^{n+1}}{\sqrt{D}}$, which gives the right inequality in \eqref{VBoundsNew+}.
\end{proof}

The next lemma gives lower bounds on the ratio of two successive terms in a manner similar to Lemma 4 in \cite{Luca1}.
\begin{lemma}\label{ULowerBd}
Let $x$ be a positive integer, $y$ be any nonzero integer satisfying $|y+1|\leq x$, and consider the sequence $U_{n}(x,y)$. For all integers $n>2$, the bound $$\frac{U_{n}}{U_{n-1}}\geq x-\frac{xy}{x^{2}-y}$$ applies when $y<0$. And for all integers $n\geq2$, the bound $$\frac{U_{n}}{U_{n-1}}\geq\alpha$$ applies when $y>0$.
\end{lemma}
\begin{proof}
The rational numbers $\frac{U_{n}}{U_{n-1}}$ are the convergents of the generalized continued fraction expansion of $\alpha$, $$\alpha=x+\cfrac{-y}{x+\cfrac{-y}{x+\cfrac{-y}{x+\dots}}}.$$ By Lemma \ref{Identities}\ref{id3a}, we have $\frac{U_{n}}{U_{n-1}}=\frac{U_{n-1}}{U_{n-2}}-\frac{y^{n-2}}{U_{n-2}U_{n}}$, and so when $y<0$, the even convergents (i.e. the rational numbers $\frac{U_{n}}{U_{n-1}}$ with $n$ even) are underestimates for $\alpha$ and are less than each subsequent convergent, while the odd convergents are overestimates and greater than each subsequent convergent. In addition, $\frac{U_{3}}{U_{2}}\geq\frac{U_{4}}{U_{3}}$, and so $$\frac{U_{n}}{U_{n-1}}\geq\frac{U_{4}}{U_{3}}=x-\frac{xy}{x^{2}-y}$$ for all $n\geq3$.

When $y\geq0$ and $n\geq2$, the ratios $\frac{U_{n}}{U_{n-1}}$, are all overestimates of $\alpha$.  To see this, recall Lemma \ref{Identities}\ref{id3a} again $\frac{U_{n}}{U_{n-1}}=\frac{U_{n-1}}{U_{n-2}}-\frac{y^{n-2}}{U_{n-2}U_{n}}$, and observe that $\alpha<x=\frac{U_{2}}{U_{1}}$.
\end{proof}
Let $\log_{\alpha}$ denote the logarithm base $\alpha$. We will need the following estimates, which when $x=1$ are those of Lemma 6 in \cite{Luca1}.
\begin{lemma}\label{EqBounds}
Given a solution in positive integers $(n,k,p,q,x,y)$ to equation \eqref{MainEq}, with $y<0$ satisfying $|y+1|\leq x$, we have
\begin{itemize}
\item{} $\log_{\alpha}(k)+p\log_{\alpha}\left(x\right)+(k-2)p<(n-1)q$ \text{if $k\geq2$, and}
\item{} $q\log_{\alpha}(x)+(n-2)q<(k-1)p+\log_{\alpha}\left(k\right)+\log_{\alpha}\left( 1+\frac{x^{2}-2y}{x^{p-1}\left(x^{3}-x^{2}-2xy+y\right)}\right)$\text{ provided $k\geq3$.}
\end{itemize}
\end{lemma}
\begin{proof}
The first inequality follows from bounds in Lemma \ref{UVBoundsNew} along with the fact that $kU_{k}^{p}<U_{n}^{q}$ for $k\geq2$. If we set $x=1,y=-1$ and note that $U_{k}^{p}<kU_{k}^{p}$, then we obtain the first inequality in Lemma 6 of \cite{Luca1}. 

To prove the second inequality, we apply Lemma \ref{ULowerBd} to obtain the inequality 
\[
\frac{U_{k}}{U_{k-i}}=\prod_{j=0}^{i-1}\left(\frac{U_{k-j}}{U_{k-j-1}}\right)\geq x\left(\frac{x^{3}-2xy}{x^2-y}\right)^{i-1},
\]
which applies for all $1\leq i\leq k-1$ provided $k\geq3$. Set $\tau:=\frac{x^2-y}{x^{3}-2xy}$. From Lemma \ref{UVBoundsNew} we have
\begin{align*}
x^q\alpha^{\left(n-2\right)q}<&U_{n}^{q}=kU_{k}^{p}\sum_{j=0}^{k-1}\frac{k-j}{k}\left(\frac{U_{k-j}}{U_{k}}\right)^{p}\\
<&kU_{k}^{p}\left( 1+x^{-p}+x^{-p}\tau^{p}+x^{-p}\tau^{2p}+\dots+x^{-p}\tau^{p(k-2)}\right)\\
<&kU_{k}^{p}\left(1+\frac{1}{x^{p}}\left(\frac{x^{3}-2xy}{x^{3}-x^2-2xy+y}\right)\right)\\
\end{align*}
Taking logarithms gives the result in the statement of the lemma.
\end{proof}

\begin{lemma}\label{Y1EqBounds}
Given a solution in positive integers $(n,k,p,q,x,y)$ to equation \eqref{MainEq} with $y>0$ satisfying $|y+1|\leq x$, we have
\begin{itemize}
\item{} $\log_{\alpha}(k)+p\log_{\alpha}(x)+(k-2)p<q\left(n-\log_{\alpha}\left(\sqrt{D}\right)\right)$ \text{if $k\geq2$, and}
\item{} $q\log_{\alpha}(x)+(n-2)q<\log_{\alpha}(k)+p\left(k+1-\log_{\alpha}\left(\sqrt{D}\right)\right)-\log_{\alpha}\left(\alpha^{p}-1\right)$\text{ provided $k\geq3$.}
\end{itemize}
\end{lemma}

\begin{proof}
For a solution to \eqref{MainEq} with $k\geq2$, we must have $U_{k}^{p}<U_{n}^{q}$, and so $kx^{p}\alpha^{p(k-2)}<kU_{k}^{p}<U_{n}^{q}<D^{-q/2}\alpha^{qn}.$ Taking logarithms, the first inequality follows.

We may prove the second inequality in a manner similar to the one employed in Lemma \ref{EqBounds}. Set $\tau=\frac{1}{\alpha}$. By Lemma \ref{ULowerBd} we have $$\frac{U_{k}}{U_{k-i}}=\prod_{j=0}^{i-1}\frac{U_{k-j}}{U_{k-j-1}}\geq\alpha^{i},$$ applies for all $1\leq i\leq k-1$, provided $k\geq2$, it follows that
\begin{align*}
x^{q}\alpha^{(n-2)q}<&U_{n}^{q}=kU_{k}^{p}\sum_{j=0}^{k-1}\frac{k-j}{k}\left(\frac{U_{k-j}}{U_{k}}\right)^{p}\\
<&kU_{k}^{p}\left(1+\tau^{p}+\tau^{2p}+\dots+\tau^{(i-1)p}\right)\\
<&kU_{k}^{p}\frac{1}{1-\tau^{p}}
\end{align*}
As before, taking logarithms gives the result.
\end{proof}

The following result is an amalgamation of part of Theorem 1.5 in \cite{Sanna} and Lemma 3.1 in the same work. It will allow us to determine the precise divisibility of $U_{n}$ by powers of $D$, which will be useful in the proof of Lemma \ref{SecondKindEquation}.
\begin{lemma}[Sanna]\label{pValuation}
For any prime $p_{i}$, let $\nu_{p_{i}}\left(n\right)$ denote the $p_{i}$-adic valuation of the integer $n$. Suppose that $p_{1},p_{2},\dots,p_{m}$ are the primes dividing $D$. Then $$\nu_{p_{i}}\left({U_{n}}\right)=\begin{cases}\nu_{p_{i}}\left(n\right)+\nu_{p_{i}}\left(U_{p_{i}}\right)-1&\text{ if $p_{i}\mid n$,}\\ 0&\text{ if $p_{i}\nmid n$.}\end{cases}$$
Moreover, if $p_{i}\geq5$, then $\nu_{p_{i}}\left(U_{p_{i}}\right)=1$.
\end{lemma}
Assuming that neither $2$ nor $3$ divide $n$ and writing the prime factorization of $D$, $$D=p_{1}^{e_{1}}\cdot\dots\cdot p_{m}^{e_{m}},$$ we may use $\nu_{D}(N)=\min\limits_{1\leq i\leq m}{\lfloor \nu_{p_{i}}(N)/e_{i}\rfloor}$ and Lemma \ref{pValuation} to obtain
\begin{align*}
\nu_{D}(U_{n})=&\min_{1\leq i\leq m}{\lfloor \nu_{p_{i}}(U_{n})/e_{i}\rfloor}\\
=& \min_{1\leq i\leq m}{\lfloor \nu_{p_{i}}(n)/e_{i}\rfloor}\\
=&\nu_{D}(n).
\end{align*}
The following result extends Lemma 8 in \cite{Luca1}.
\begin{lemma}\label{SecondKindEquation}
If $(p,q,k,x)$ is a solution in positive integers to $$V_{p}(x)k^{2}+(V_{p}(x)-2)k-1=\pm D^{p-q}V_{p}(x)^{2},$$ with $p$ odd and $y=-1$, then $(p,q,k,x)=(1,1,1,1),(1,1,2,1),(1,1,2,5)$.
\end{lemma}
\begin{proof}
First, we note that $V_{p}(x)^2$ can never be a multiple of $D$, 
and since $D^{p-q}V_{p}(x)^{2}$ is an integer, we must have $p-q\geq0$. If $p=1$, then $q=1$ and the equation becomes $$xk^2+(x-2)k-1=\pm x^2.$$ We can rearrange this to get $$x(k^2+k\mp x)=2k+1.$$ From this, we see that both factors on the left side are less than or equal to $2k+1$. That is
\begin{align*}
x\leq&2k+1\text{, and}\\
k^2+k\mp x\leq&2k+1.
\end{align*}
If the sign on the left is a $+$, then we obtain $$x\leq -k^2+k+1,$$ which leaves $x=1,k=1$ as the only possibility. If the sign on the left is a $-$, then we obtain $$k^2-k-1\leq x\leq 2k+1,$$ and so $k^{2}-3k-2\leq0$, which means we must have $k<4$. A brief check reveals $k=2,x=1$ and $k=2,x=5$ as the only two possibilities, and so we have obtained all solutions listed in the statement of the lemma.

It remains to show that there are no additional solutions. When $p\geq3$, the left hand side is always positive, and so the sign on the right is a $+$. The cases with $p\in\{3,5\}$ and $x\in\{1,2\}$ may be checked by substituting those particular values into the equation and solving the resulting quadratic in $k$. No additional solutions arise in those cases. Hereafter, we assume that $x\geq 3$ or $p\geq7$. Set $r=p-q$. From $$V_{p}(x)k^2+\left(V_{p}(x)-2\right)k-1=D^{r}V_{p}(x)^2,$$ we see that $V_{p}(x)$ divides $2k+1$, which means that $V_{p}(x)$ is odd, and so $x$ must be odd and $p$ cannot be a multiple of $3$. Moreover, $2k+1=aV_{p}(x)$ holds for the odd integer $a=k^2+k-D^rV_{p}(x)$, and so $k=(aV_{p}(x)-1)/2.$ Substituting, we get 
\begin{equation}\label{NewEq}
a^2V_{p}(x)^2-(4a+1)=4D^rV_{p}(x).
\end{equation}
Some manipulation of equation \eqref{NewEq} allows us to obtain
\begin{align*}
a^2V_{p}(x)^{4}-4aV_{p}(x)^{2}+4=&4+V_{p}(x)^{2}(4D^{r}V_{p}(x)+1)\\
=&4D^{r}V_{p}(x)^{3}+(V_{p}(x)^{2}+4)\\
=&4D^rV_{p}(x)^3+DU_{p}(x)^{2},
\end{align*}
and so
\begin{equation}\label{EqualtoSquareEqn}
4D^rV_{p}(x)^3+DU_{p}(x)^{2}=m^2
\end{equation}
for $m=aV_{p}(x)^2-2$. Equation \eqref{NewEq} also allows us to see that $V_{p}(x)\mid(4a+1)$, and so $a\geq(V_{p}(x)-1)/4\geq7$. Since $V_{p}(x)\geq29$, we have $4a+1<5a<a^2V_{p}(x)^{2}/2$. It follows that $$4D^rV_{p}(x)=a^{2}V_{p}(x)^{2}-(4a+1)>\frac{a^{2}V_{p}(x)^{2}}{2},$$ which gives us $$a<\frac{2^{3/2}D^{r/2}}{V_{p}(x)^{1/2}},$$ and since $a\geq\frac{V_{p}-1}{4}$, we get 
\begin{equation}\label{RootVInequality}
(V_{p}(x)-1)V_{p}(x)^{1/2}<2^{7/2}D^{r/2}.
\end{equation}
Now let $c,d$ be such that $D^{c}\mid\mid U_{p}(x)$ and $D^{d}\mid\mid m$. From \eqref{EqualtoSquareEqn} and Lemma \ref{pValuation}, we must have $U_{p}(x)=D^{c}u$ and $m=D^{d}v$ for some integers $u,v$ with both $\gcd\left(u,D\right)<D$ and $\gcd\left(v,D\right)<D$. Note that Lemma \ref{pValuation} applies here because we are assuming that $p$ is odd, so $2\nmid p$ and we have established $V_{p}(x)$ is odd, so $3\nmid p$. Moreover, since we have established that $x$ is odd, we must have that $D$ is odd, so $\gcd\left(4V_{p}(x)^{3},D\right)=1.$ From $$4D^rV_{p}(x)^{3}+D^{2c+1}u^{2}=D^{2d}v^{2},$$ and since $D\nmid v$, we see that either $r=2c+1$ or $r=2d$ and in either case, $$r\leq 2c+1=2\nu_{D}(U_{p}(x))+1=2\nu_{D}(p)+1\leq 2\log_{D}p+1,$$ where the equality follows from Lemma \ref{pValuation}.
Hence $D^{r}\leq Dp^{2}$. From \eqref{RootVInequality}, we then have that $$V_{p}(x)^{1/2}\left(V_{p}(x)-1\right)\leq2^{7/2}D^{r/2}\leq 2^{7/2}D^{1/2}p,$$ and since from the inequalities \eqref{VBoundsNew-} in Lemma \ref{UVBoundsNew}, $V_{p}(x)\geq x\alpha^{p-1}$, it follows that $$x^{1/2}\alpha^{\left(p-1\right)/2}\left(x\alpha^{p-1}-1\right)\leq 2^{7/2}D^{1/2}p=2^{7/2}\left(2\alpha-x\right)p,$$ which is false whenever $x\geq3,p\geq3$, when $x=2,p\geq5$ and when $x=1,p\geq7$ as in \cite{Luca1}. It follows that there are no more solutions.
\end{proof}

\section{The case $k=2$}\label{k2Case}

In \cites{Luca1,Luca2}, it was assumed that $k\geq3$, which may be done when working exclusively with the Fibonacci numbers since the fact that $F_{1}=F_{2}=1$ makes solving the equation incredibly easy in the case $k=2$. However, the more general family of Lucas sequences under consideration requires a little care in this situation. We list all the solutions to \eqref{MainEq} with $k=2$ and $y=\pm1$ in the Proposition \ref{MainProposition}. Some comments on the choice of restriction on $y$ will follow the proof.
\begin{proposition}\label{MainProposition}
The only solutions to the equation $1+2x^p=U_{n}(x,y)^{q}$ in positive integers $n,p,q,x$ and with $y=\pm1$ are $$(n,p,q,x,y)=(4,p,1,1,-1), (1+2^{p+1},p,1,2,1) ,(3,2,2,2,1) ,(3,1,1,2,-1) \text{ and } (5,3,1,3,1).$$
\end{proposition}
\begin{proof}
When $x=y=1$, the equation becomes $3=U_n(1,1)^{q}$, which has no solutions since $U_{n}(1,1)\in\{0,\pm1\}$ for all $n$. When $x=1,y=-1$ there is exactly one solution for any choice of $p$, that being when $n=4$ and $q=1$. This gives the solutions $(n,p,q,x,y)=(4,p,1,1,-1)$ listed in the statement of the proposition. We will henceforth assume $x>1$.

When $x=2$, the equation becomes $1+2^{p+1}=U_{n}^{q}$. By Mih\u{a}ilescu's Theorem (Catalan's Conjecture), the only solution to this equation with $q>1$ is $(n,p,q,x,y)=(3,2,2,2,1)$. When $q=1$ and $y=1$, we make the further observation that since $U_{n}(2,1)=n$, there is exactly one solution for every choice of $p$, that being $n=1+2^{p+1}$, giving the solutions $(n,p,q,x,y)=(1+2^{p+1},p,1,2,1)$ listed above. This allows us to assume that $x>2$ when $y=1$ for the remainder of the proof.

When $n$ is even, $U_{n}$ is divisible by $x$, so by fixing $n$ and reducing $1+2x^p=U_{n}^{q}$ modulo $x$, we see that $n$ must be odd. Moreover, since $U_{n}^{q}-1$ must be even, it follows that $U_{n}$ must be odd.

When $n=1$, we get $1+2x^p=1$, to which there are no solutions with $x>0$. When $n=3$ and $q=1$, we have $1+2x^{p}=x^{2}-y$. When $y=1$, this gives $2(1+x^p)=x^2$, which means that $x$ is even and $\frac{x^{2}}{2}=1+x^{p}$ is odd, which is impossible, and so there are no solutions with $n=3,y=1$ and $q=1$. When $y=-1$, we have $2x^{p}=x^{2}$, whose only solution is $p=1,x=2$, from which we obtain the solution $(n,p,q,x,y)=(3,1,1,2,-1)$ listed in the statement of the proposition.

When $n\geq5$ and $p\leq3$, we have $1+2x^{p}\leq U_{5}(x,y)$ with equality holding only when $(n,p,q,x,y)=(5,3,1,3,1)$, which is the last of the listed solutions. So we may assume $p\geq4$ whenever $n\geq5$.

\vspace*{2mm}
Now suppose that $2x^p=U_{n}^{q}-1$ and so, consequently, any prime divisor of $U_{n}^{q}-1$ must also divide $2x$. Consider the sequence $S$ with terms $S_{i}=\{U_{n}^{i}-1\}_{i=1}^{\infty}$ and observe that since $n$ is odd, every odd prime dividing $2x$ must divide either $U_{n}-1$ or $U_{n}+1$, meaning it must divide either $S_{1}$ or $S_{2}$, and since $U_{n}$ is odd, $2$ divides $S_{i}$ for every $i$. Zsigmondy's Theorem states that all but a possible handful of exceptional terms in the sequence $S$ have a primitive prime divisor. Since we have shown that every prime dividing $S_{q}=2x^p$ must enter the sequence $S$ in either the first or second term, it follows that either $q=1$, $q=2$ or, as in the sole applicable exception given by Zsigmondy's Theorem, $q=6$, $U_{n}=2$. This latter case may be ruled out immediately however, since we require $U_{n}^{q}-1$ to be even.

In the case $q=2$, we have that either $x^{2}\mid U_{n}-1$, as when $y=-1$ or when $y=1$ and $n\equiv1\pmod{4}$, or $x^{2}\mid U_{n}+1$, as when $y=1$ and $n\equiv3\pmod{4}$. In the first case, this means that $U_{n}+1\equiv2\pmod{x^{2}}$, and in the second $U_{n}-1\equiv-2\pmod{x^{2}}$. Since we also have $2x^{p}=(U_{n}-1)(U_{n}+1)$, it follows that $x\mid2$, which has already been dealt with for $q>1$.

\vspace*{2mm}

This leaves only the case $q=1$. Note that due to our earlier work, we may assume that $n\geq5$ and $p\geq4$, and that $x\geq3$ when $y=1$. Using Lemma \ref{Identities} part \ref{id3}, write 
\begin{equation*}
U_{n}=U_{\frac{n-1}{2}}V_{\frac{n+1}{2}}-(y)^{\frac{n+1}{2}}U_{-1}=U_{\frac{n+1}{2}}V_{\frac{n-1}{2}}-(y)^{\frac{n-1}{2}}U_{1}.
\end{equation*}
Since $U_{-1}=-y$, we obtain
\begin{equation*}
U_{n}-1=U_{\frac{n-1}{2}}V_{\frac{n+1}{2}}
\end{equation*}
when $y=1$, and
\begin{equation*}
U_{n}-1=\begin{cases}
U_{\frac{n-1}{2}}V_{\frac{n+1}{2}}\quad\text{if $n\equiv1\pmod{4}$,}\\
U_{\frac{n+1}{2}}V_{\frac{n-1}{2}}\quad\text{if $n\equiv3\pmod{4}$}
\end{cases}
\end{equation*}
when $y=-1$.

The relation $U_{n}-1=2x^{p}$ means that the only primes that may divide $U_{\frac{n-1}{2}}$ (or respectively $U_{\frac{n+1}{2}}$ when $y=-1$ and $n\equiv3\pmod{4}$) are $2$ or the prime divisors of $x$. When $y=-1$, or when $y=1$ and $x\geq3$, the discriminant $D=x^{2}-4y$ is positive, and so $\alpha$ and $\beta$ are real. By Carmichael's Theorem \cite{Car} (or the more general Primitive Divisor Theorem \cite{B-H-V}), the terms $U_{\frac{n-1}{2}}$ and $U_{\frac{n+1}{2}}$ possess a primitive divisor (relative to the sequence $\{U_{j}\}_{j=1}^{\infty}$) whenever $\frac{n-1}{2}>12$. Since the rank of apparition of $x$ is $2$, and the rank of apparition of $2$ must be either $2$ or $3$, it follows that $n\leq25$.

Using Lemma \ref{Identities} part \ref{id5}, write $$1+2x^{p}=x^{n-1}-(n-2)x^{n-3}y+\dots+(-1)^{\frac{n-3}{2}}\frac{(n-1)(n+1)}{8}x^{2}y^{\frac{n-3}{2}}+(-1)^{\frac{n-1}{2}}y^{\frac{n-1}{2}}.$$ 
When $y=1$ and $n\equiv3\pmod{4}$, we obtain $$2+2x^{p}=x^{n-1}-(n-2)x^{n-3}y+\dots+(-1)^{\frac{n-3}{2}}\frac{(n-1)(n+1)}{8}x^2y^{\frac{n-3}{2}},$$ giving $x^2\mid2$, which is impossible.

When $y=-1$, or when $y=1$ and $n\equiv1\pmod{4}$, this gives $$2x^{p}=x^{n-1}-(n-2)x^{n-3}y+\dots-\frac{(n-1)(n+1)}{8}x^2y.$$ Factoring an $x^2$ from the right side, which remains an integer polynomial in $x^{2}$, 
we also see that $x^{2}$ must divide $\frac{(n-1)(n+1)}{8}$, which gives $x^2\leq 78$, and so $x\leq8$. 
A brief check in SageMath \cite{sage} reveals that no additional solutions arise in these cases, which completes the proof. 

\end{proof}

The reader may have observed that our specific use of both Zsigmondy's Theorem, and the Primitive Divisor Theorem rely on the fact that $y=\pm1$. Indeed, when invoking Zsigmondy's Theorem, we relied on the fact that $U_{n}^{2}-1\equiv0\pmod{x}$. If we consider for a moment $y\neq1$, we would have that $0\equiv U_{n}^{q}-1\equiv(-y)^{\frac{q(n-1)}{2}}-1\pmod{x^2}$ when $p\geq2$ (this can be seen from Lemma \ref{Identities} part \ref{id5}), and we cannot immediately tell the first value of $q$ for which $(-y)^{\frac{q(n-1)}{2}}\equiv1\pmod{x^{2}}$.

Similarly, once we had reduced the problem to the case $q=1$ and we were looking to apply the Primitive Divisor Theorem, we used Lemma \ref{Identities} part \ref{id3}, but in the case of $|y|>1$, this does not allow us to write $2x^{p}$ as a product. This suggests that the methods applied in Proposition \ref{MainProposition} are specific to the cases $y=\pm1$, and different methods would need to be explored in order to expand this result.

\section{Proof of Theorem 1}\label{MainSec}
The proof proceeds in a similar manner to section 2 of \cite{Luca1}. Assume hereafter that $k\geq3$. We have the following proposition.
\begin{proposition}
Let $(n,k,p,q,x,y)$ be a solution to \eqref{MainEq} with $n,k,p,q$ and $x$ positive integers and $y\in\{\pm1\}$ with $D=x^2-4y>0$ (i.e. $x>2$ if $y=1$). Then
\begin{equation}
\frac{D^{(q-p)/2}(k\alpha^{p}-(k+1))}{(\alpha^{p}-1)^{2}}\alpha^{p(k+1)}-\alpha^{qn}=D^{q/2}R_{2}-D^{q/2}R_{1}-\frac{D^{(q-p)/2}\alpha^{p}}{(\alpha^{p}-1)^{2}},
\end{equation}
where $$|R_{1}|<k^{2}\alpha^{k(p-2)},\quad\text{and}\quad |R_{2}|<\alpha^{n(q-2)}.$$
\end{proposition}
\begin{proof}
From $$U_{j}^{p}=\frac{(\alpha^{j}-\beta^{j})^{p}}{D^{p/2}}=\frac{\alpha^{jp}}{D^{p/2}}+\zeta_{p,j}$$ with $$|\zeta_{p,j}|<\frac{2^{p}\alpha^{(p-2)j}}{D^{p/2}}<\alpha^{j(p-2)},$$ we have 
$$\sum\limits_{j=1}^{k}jU_{j}^{p}=\frac{1}{D^{p/2}}\left(\sum\limits_{j=1}^{k}j\alpha^{jp}\right)+R_{1},$$ with $|R_{1}|<k^2\alpha^{k(p-2)}.$ Hence $$\sum\limits_{j=1}^{k}jU_{j}^{p}=\frac{k\alpha^{p}-(k+1)}{D^{p/2}(\alpha^{p}-1)^{2}}\alpha^{p(k+1)}+\frac{\alpha^{p}}{D^{p/2}(\alpha^{p}-1)}+R_{1}.$$

Similarly, we have $$U_{n}^{q}=\frac{\alpha^{nq}}{D^{q/2}}+R_{2}$$ with $$|R_{2}|<\alpha^{n(q-2)}.$$ The statement of the proposition follows.
\end{proof}
In both the cases $y=1$ and $y=-1$, the bounds $$D^{q/2}|R_{1}|<D^{q/2}k^{2}\alpha^{kp-2k}$$  and $$\frac{D^{(q-p)/2}\alpha^{p}}{(\alpha^{p}-1)^{2}}\leq D^{q/2}\alpha^{3}\leq D^{q/2}\alpha^{kp-k+3}$$ both apply. In order to bound the term $5^{q/2}R_{2}$, we must work with the cases $y=\pm1$ separately.
\begin{lemma}\label{BoundR2}
When $y=-1$, the bound $$D^{q/2}|R_{2}|<k\left(D+\frac{1}{x}\right)^{q+1}\alpha^{kp-2k/q+1}$$ applies.
\end{lemma}

\begin{proof}
First, observe that $\sqrt{D}(\alpha^{2}-\beta^{2})=Dx$, and so $\sqrt{D}\alpha^{2}=Dx+\sqrt{D}\beta^{2}$. Since $\sqrt{D}\beta^{2}=\frac{\sqrt{D}}{\alpha^{2}}<1$ for all $x\geq1$, we obtain $\sqrt{D}\alpha^{2}<Dx+1$, from which the inequality $$\left(\frac{\sqrt{D}\alpha^{2}}{x}\right)^{q}<\left(D+\frac{1}{x}\right)^{q}$$ follows. Examining the term $D^{q/2}|R_{2}|$, we form the inequality
\begin{align*}
D^{q/2}|R_{2}|<&D^{q/2}\alpha^{n(q-2)}\\
=&D^{q/2}\alpha^{(n-2)q+2q-2n}.\\
\end{align*}
The second inequality in Lemma \ref{EqBounds} gives $(n-2)q<(k-1)p+\log_{\alpha}{k}+\log_{\alpha}\left(1+\frac{x^{2}+2}{x^{p-1}(x^{3}-x^{2}+2x-1)}\right)-q\log_{\alpha}x$, where $1+\frac{x^{2}+2}{x^{p-1}(x^{3}-x^{2}+2x-1)}<4$, it follows that
\begin{align*}
D^{q/2}|R_{2}|<&D^{q/2}\cdot 4k\alpha^{(k-1)p+2q-2n}x^{-q}\\
<&4k\left(D+\frac{1}{x}\right)^{q}\alpha^{(k-1)p-2n}.
\end{align*}

The first inequality in Lemma \ref{EqBounds} gives $$n>\frac{k-2}{q}+1,$$ which yields
\begin{align*}
D^{q/2}|R_{2}|<&4k\left(D+\frac{1}{x}\right)^{q}\alpha^{(k-1)p-2(k-2)/q-2}.\\
\end{align*}
Since $(k-1)p-\frac{2(k-2)}{q}-2=kp-\frac{2k}{q}+1-\left(p-\frac{4}{q}+3\right)$, where $p-\frac{4}{q}+3>0$ for positive $p$ and $q$, and $4<D+\frac{1}{x}$ for all positive $x$, we obtain
\begin{align*}
D^{q/2}|R_{2}|<&k\left(D+\frac{1}{x}\right)^{q+1}\alpha^{kp-2k/q+1},
\end{align*}
as in the statement of the lemma.
\end{proof}
\begin{lemma}
When $y=1$, the bound $$D^{q/2}|R_{2}|<k\left(V_{2}+\frac{1}{x}\right)^{q+1}\alpha^{kp-2k/q+1/2}$$ applies.
\end{lemma}

\begin{proof}
From
\begin{align*}
xV_{2}+1=x^{3}-2x+1>&x^{3}-3x=\frac{2x^{3}-6x}{2}=\frac{x(x^{2}-2)+x(x^{2}-4)}{2}\\
=&\frac{x(x^{2}-2)+xD}{2}>\sqrt{D}\left(\frac{(x^{2}-2)+x\sqrt{D}}{2}\right)\\
=&\sqrt{D}\alpha^{2},
\end{align*}
we get $$\frac{\sqrt{D}\alpha^{2}}{x}<\left(V_{2}+\frac{1}{x}\right).$$
The second inequality in Lemma \ref{Y1EqBounds} gives  
\begin{align*}
D^{q/2}|R_{2}|<&D^{q/2}\alpha^{n(q-2)}\\
<&D^{q/2}\cdot\frac{k}{\alpha^{p}-1}x^{-q}D^{-p/2}\alpha^{p(k+1)+2q-2n}.\\
\end{align*}
Since $\alpha^{p}>\alpha\geq\frac{3+\sqrt{5}}{2}>2$, and $\frac{1}{c-1}<\frac{2}{c}$ whenever $c>2$, we have $\frac{k}{\alpha^{p}-1}<\frac{2k}{\alpha^{p}}$. Since $D=x^2-4>x>\frac{x+\sqrt{x^{2}-4}}{2}=\alpha$ for $x\geq3$, we can write $D^{-p/2}<\alpha^{-p/2}$. Combining this with $\frac{\sqrt{D}\alpha^{2}}{x}<\left(V_{2}+\frac{1}{x}\right)$, we get
\begin{align*}
D^{q/2}|R_{2}|<&D^{q/2}\cdot 2k\alpha^{(k-\frac{1}{2})p+2q-2n}x^{-q}\\
<&2k\left(V_{2}+\frac{1}{x}\right)^{q}\alpha^{(k-\frac{1}{2})p-2n}.
\end{align*}

The first inequality in Lemma \ref{Y1EqBounds} gives $$\log_{\alpha}k+p\log_{\alpha}x+(k-2)p<q(n-\log_{\alpha}\sqrt{D}).$$ Since $D>\alpha$ for $x\geq3$, we have $\frac{1}{2}\log_{\alpha}D>\frac{1}{2}$. Moreover, since $x>\alpha$, $p\log_{\alpha}x+(k-2)p>(k-1)p$, so we can write $$p(k-1)<q\left(n-\frac{1}{2}\right),$$ which leads to $$2n>2\left(\frac{k-1}{q}\right)+1.$$ We then write
\begin{align*}
D^{q/2}|R_{2}|<&2k\left(V_{2}+\frac{1}{x}\right)^{q}\alpha^{(k-\frac{1}{2})p-2(k-1)/q-1}.\\
\end{align*}
We can bound the exponent $(k-\frac{1}{2})p-\frac{2(k-1)}{q}-1=kp-\frac{2k}{q}-\left(\frac{p}{2}-\frac{2}{q}+1\right)$ by observing that $\frac{p}{2}-\frac{2}{q}+1\geq-\frac{1}{2}$ for positive integers $p$ and $q$. So 
\begin{align*}
D^{q/2}|R_{2}|<&k\left(V_{2}+\frac{1}{x}\right)^{q+1}\alpha^{kp-2k/q+1/2},
\end{align*}
as in the statement of the lemma.
\end{proof}

In both cases, $D^{q/2}|R_{2}|<4k(D+\frac{1}{x})^{q+1}\alpha^{kp-2k/q+2}$.

\vspace*{2mm}

Set $\Delta_{p}:=(\alpha^{p}-1)^{2}$, $C=D+\frac{1}{x}$ and $$z_{q}(k):=D^{q/2}\alpha^{3}+4kC^{q+1}\alpha^{2}+D^{q/2}k^{2}.$$ Then we have 
\begin{equation}\label{absineq}
\left|\frac{D^{(q-p)/2}(k\alpha^{p}-k-1)}{\Delta_{p}}-\alpha^{qn-p(k+1)}\right|<\frac{z_{q}(k)}{\alpha^{k/q+p}}.
\end{equation}
Now set $\mu:=qn-p(k+1)$. Suppose first that $$\alpha^{\mu}\leq\frac{D^{(q-p)/2}}{3\Delta_{p}}.$$ When $k\geq3$ and $p\geq1$, we have $$\frac{D^{(q-p)/2}}{\Delta_{p}}(k\alpha^{p}-k-1)-\alpha^{\mu}\geq \frac{D^{(q-p)/2}}{\Delta_{p}}(k(\alpha^{p}-1)-1-1/3)>\frac{D^{(q-p)/2}}{\Delta_{p}}.$$ Hence $$\alpha^{k/q+p}<D^{(p-q)/2}\Delta_{p}z_{q}(k).$$

Suppose now that
\begin{equation}\label{Opposite}
\alpha^{\mu}>\frac{D^{(q-p)/2}}{3\Delta_{p}}.
\end{equation} 
If the left hand side of \eqref{absineq} is zero, we may rearrange and take norms in $\mathbb{Q}(\sqrt{D})$ to get $$k^{2}y^{p}-(k+1)kV_{p}+(k+1)^{2}=-D^{p-q}y^{qn-p(k+1)}(y^{p}-V_{p}+1)^{2}.$$ If $y=1$ or $y=-1$ and $p$ is even, we may solve $$(V_{p}-2)k^{2}+(V_{p}-2)k-1=\pm D^{p-q}(V_{p}-2)^{2}.$$ If $q>p$ then if a positive integer $l$ divides $D^{p-q}(V_{p}-2)^{2}$, reducing modulo $l$ gives $-1\equiv0\pmod{l}$, which means that $\pm D^{p-q}(V_{p}-2)^{2}=\pm1$. From this, we obtain $$(V_{p}-2)k^{2}+(V_{p}-2)k-1=\pm1,$$ which has only the trivial solution $p=1,x=3,k=1$ when $y=1$ and $x\geq3$, and has no solutions with $y=-1$ and even $p\geq2$. 

If $p\geq q$ then $V_{p}-2$ divides $1$, which leaves only the possibilities
\begin{itemize}
\item{}$y=-1,x=1,p=2,$ in which case we have $q=2$, $k=1$ or $q=1$, $k=2$, both of which yield trivial solutions to equation \eqref{FibonacciEq};
\item{} $y=1,x=3,p=1$, in which case $k=1$, $q=1$ is the only possibility, another trivial solution.
\end{itemize}

When $y=-1$ and $p$ is odd, Lemma \ref{SecondKindEquation} tells us that there are no solutions to $$V_{p}k^{2}+(V_{p}-2)k-1=\pm D^{p-q}V_{p}^{2}$$ with $x\geq2$ and $k\geq3$.

If the left hand side of \eqref{absineq} is nonzero, then $$|k\alpha^{p}-(k+1)-\alpha^{\mu}D^{(p-q)/2}\Delta_{p}|<\frac{D^{(p-q)/2}\Delta_{p}z_{q}(k)}{\alpha^{k/q+p}}.$$ Following from \eqref{Opposite} and the fact that $\sigma(\Delta_{p})=(\beta^{p}-1)^{2}\leq2$, $$|\beta|^{\mu}=\alpha^{-\mu}<3D^{(p-q)/2}\Delta_{p},$$ whence we also have $$|k\beta^{p}-(k+1)-\beta^{\mu}(-D^{1/2})^{p-q}\sigma(\Delta_{p})|<(2k+1)+6D^{p-q}\Delta^{p}.$$ Multiplying the two left hand sides we get the norm of a nonzero algebraic integer, which is $\geq1$, so the bound 
\begin{equation}\label{kBound}
\alpha^{k/q+p}<D^{(p-q)/2}\Delta_{p}z_{q}(k)((2k+1)+6D^{p-q}\Delta_{p})
\end{equation}
applies, which is weaker than the earlier bound. This provides a general upper bound on $k$ for fixed $p$ and $q$.

The following result shows that \eqref{kBound} is sufficient to bound all sequences $U_{n}(x,\pm1)$.
\begin{lemma}\label{kBound}
If $(n,k,p,q,x)$ is a positive integer solution to equation \eqref{MainEq} then $x\leq\frac{(k+1)^{2}}{4}+1$.

\end{lemma}
\begin{proof}
Using Lemma \ref{Identities} part \ref{id5}, we see that $U_{j}^{p}$ is a polynomial in $x$ with constant term $0$ when $j$ is even and $(-y)^{p\left(\frac{j-1}{2}\right)}$ if $j$ is odd. Defining $\epsilon=0$ if $n$ is even and $\epsilon=(-y)^{q\left(\frac{n-1}{2}\right)}$ if $n$ is odd, we see that the expression
\begin{equation}\label{xBound}
\left(\sum_{j=1}^{k}jU_{j}^{p}\right)-U_{n}^{q}
\end{equation}
is a polynomial in $x$ with constant term $$\left(\sum_{j\text{ odd}\leq k}j(-y)^{p\left(\frac{j-1}{2}\right)}\right)-\epsilon\leq\frac{(k+1)^{2}}{4}+1.$$

In any solution to \eqref{MainEq}, $x$ must divide this term. It follows that $x\leq\frac{(k+1)^{2}}{4}+1$, as in the statement of the lemma.
\end{proof}
Taking logarithms in \eqref{kBound}, we obtain
\begin{equation}\label{kBound1}
k<q\left(\frac{\log{\left(D^{(p-q)/2}\Delta_{p}z_{q}(k)((2k+1)+6D^{p-q}\Delta_{p})\right)}}{\log{\alpha}}\right)-pq.
\end{equation}
Fixing $p$ and $q$ and substituting the bound $x\leq\frac{(k+1)^{2}}{4}+1$, we see that the numerator on the right side is of the form $\log{\left(g(k)\right)}$ for $g(k)$ a polynomial in $k$ with degree depending on $p$ and $q$. Given that we are assuming that $x\geq3$ when $y=1$ and that $x\geq2$ when $y=-1$, the term $\log{\alpha}$ in the denominator is no less than $\log{\left(\frac{3+\sqrt{5}}{2}\right)}$ when $y=1$ and no less than $\log{\left(1+\sqrt{2}\right)}$ when $y=-1$, and so after fixing $p$ and $q$, a bound on $k$ may be obtained that holds for every one of these sequences. 
Using SageMath, we find a bound on $k$ for each triple $(p,q,y)$ with $1\leq p,q\leq 11$ and $y\in\{\pm1\}$. The largest of these bounds, which occurred when $p=q=11$ and $y=-1$, was $15711$. Observe that our pursuit of generality has resulted in a looser bound in the case of the Fibonacci sequence than that obtained in \cite{Luca1}.

In order to reduce the computation required, we revisit the inequalities first obtained in Lemmas \ref{EqBounds} and \ref{Y1EqBounds}. The inequalities in Lemma \ref{EqBounds} imply 
$$\log_{\alpha}(k)-\epsilon_{x,p}<q(n-1)-(k-1)p<\log_{\alpha}(k)+\delta_{x,p,q},$$ where $$\epsilon_{x,p}=p-p\log_{\alpha}(x)$$ and $$\delta_{x,p,q}=q-q\log_{\alpha}(x)+\log_{\alpha}\left(1+\frac{x^{2}+2}{x^{p-1}(x^3-x^2+2x-1)}\right).$$
Using SageMath \cite{sage}, we fix $y=-1$, fix $p$ and $q$ and search for integers $k,x$ and $n$ satisfying these inequalities, with $k$ up to the bound in the previous paragraph and $x$ a divisor of the corresponding constant term as described in the proof of Lemma \ref{xBound}. An initial search reveals that there are no solutions with $x\geq11$. Applying the inequalities in Lemma \ref{Y1EqBounds} in a similar manner reveals that no solutions exist with $x\geq11$ when $y=1$ and $p,q\leq11$. We then use $3\leq x\leq10$ when $y=1$ and $2\leq x\leq10$ when $y=-1$ in \eqref{kBound1} to reduce the bound on $k$. This time, the largest bound found was $k<838$, occurring when $p=q=11$ and $y=-1$. Using the inequalities once more, along with the new bounds on $k$, we eliminate all but $29$ possibilities for $(k,n,p,q,x,y)$. These remaining cases were checked individually and none yielded an additional solution.

\section{Some Remarks on the Cases $|y|>1$}
We will briefly discuss the difficulty that arises when applying this method to a fixed value of $y$ with $|y|>1$. In this situation, the term in Lemma \ref{kBound} would tell us that $x$ must divide $\sum\limits_{j\text{ odd}\leq k}j(-y)^{\frac{p(j-1)}{2}}$ when $n$ is even or $\sum\limits_{j\text{ odd}\leq k}j(-y)^{\frac{j-1}{2}}-y^{\frac{q(n-1)}{2}}$ when $n$ is odd. This means that we can only say that $x$ must be a divisor of some term involving $k$ in the exponent. Even were we to substitute this bound into an inequality similar to \eqref{kBound}, the advantage we gained in the case $|y|=1$, which was being able to turn \eqref{kBound1} into a bound of the form $k<\log{(g(k))}$ for $g(k)$ a polynomial in $k$, will no longer be present. This suggests that some alternative method would need to be explored in order to solve equation \eqref{MainEq} for more values of $y$.

%%%%%%%%%%%%%%%%%%%%%%%%%%%%%%%%%%%%%%%%%%%%%%%%%%%%
\section*{Acknowledgments} The authors would like to thank the anonymous referee for comments which improved the quality of this paper.

\end{document}